\documentclass{scrartcl}
\ifdefined\Kindle
\usepackage[margin=1mm,papersize={110mm,75mm}]{geometry}
\AtBeginDocument{\presetkeys{todonotes}{inline}}
\fi

\usepackage{ifluatex}

\usepackage{amsmath,amsfonts,amsthm,braket}
\usepackage{enumitem}
\usepackage[unicode]{hyperref}

\ifluatex
\usepackage{polyglossia}
\setmainlanguage{english}
\usepackage{unicode-math}

\AtBeginDocument{\directlua{require("combining.lua")}}
\newcommand{\̇}{\relax\ifmmode\dot\else\.\fi}
\def\̂ {\hat}
\def\́ {\acute}
\def\̄ {\overline}
\def\̃ {\tilde}
\def\͠{\widetilde}
\else
\usepackage{amssymb}
\usepackage[utf8]{inputenc}
\usepackage[english]{babel}
\usepackage{expl3}
\newcommand{\eps}{\varepsilon}
\fi

\usepackage{csquotes}
\usepackage[backend=biber,firstinits=true,sortcites=true]{biblatex}
\usepackage{tikz}
\usetikzlibrary{arrows,positioning}
\newcommand{\atpage}[2][at~page~]{\hyperref[#2]{#1\pageref*{#2}}}

\usepackage{datetime}
\usepackage[automark]{scrlayer-scrpage}
\ohead{\headmark}
\ihead{Complex Limit Cycles}
\ofoot{\pagemark}

\usepackage{ifdraft}
\ifdraft{
\newcommand{\draftmark}{\small {\upshape\bfseries Draft version} from \csname @date\endcsname}
\ifoot[\draftmark]{\draftmark}
\PassOptionsToPackage{draft}{todonotes}
\AtBeginDocument{\date{\today\ \currenttime}}
}{}
\usepackage[backgroundcolor=white,textsize=scriptsize,obeyDraft]{todonotes}

\AtBeginDocument{

}

\newtheorem*{maintheorem}{Main Theorem}

\newcommand{\orcid}[1]{ORCID ID: \href{http://orcid.org/#1}{#1}}
\newcommand{\MRauthor}[1]{MR Author ID: \href{http://www.ams.org/mathscinet/search/author.html?mrauthid=#1}{#1}}

\ExplSyntaxOn
\newcommand{\mynewtheorem}[1]{
\newtheorem{#1}{\tl_upper_case:n #1}
\newtheorem*{#1*}{\tl_upper_case:n #1}
\expandafter\edef\csname #1autorefname\endcsname{\tl_upper_case:n #1}
}
\ExplSyntaxOff

\mynewtheorem{theorem}
\mynewtheorem{lemma}
\mynewtheorem{proposition}
\mynewtheorem{conjecture}
\mynewtheorem{corollary}

\theoremstyle{remark}
\mynewtheorem{remark}

\theoremstyle{definition}
\mynewtheorem{definition}

\newlist{tcases}{description}{1}
\newcounter{tcasesi}
\renewcommand{\thetcasesi}{\Roman{tcasesi}}

\setlist[tcases,1]{before={\setcounter{tcasesi}{0}},font=\grefstepcounter{tcasesi}\makeatother\textbf{Case \thetcasesi: }\emph}
\makeatletter
\newcommand{\grefstepcounter}[1]{\refstepcounter{#1}\xdef\@currentlabel{\@currentlabel}}
\makeatother

\title{Cheap Complex Limit Cycles}
\author{
    N. Goncharuk%
    \thanks{Cornell University, College of Arts and Sciences, Department of Mathematics, 310 Mallot Hall, Ithaca, NY, 14853, US}
    \thanks{Higher School of Economics, Department of Mathematics, 20 Myasnitskaya street, Moscow 101000, Russia}
    \thanks{Research supported by RFBR project 16-01-00748-a}
    \thanks{\orcid{0000-0002-4270-0510}, \MRauthor{978548}}
    \and
    Yu. Kudryashov%
    \footnotemark[1]
    \footnotemark[2]
    \footnotemark[3]
    \thanks{\orcid{0000-0003-4286-9276}, \MRauthor{914251}}
}

\makeatletter
\hypersetup{pdftitle=\@title,pdfauthor={Nataliya Goncharuk, Yury Kudryashov}}
\makeatother

\tikzset{
    every picture/.style={
        x=22mm,
        y=3cm,
        z={(-3mm,-8mm)},
        thick,
        every node/.append style={fill/.default=white},
        beta/.style={semithick,dash dot},
        invisible/.style={very thin,dashed},
        horizontal arc/.style={
            x radius=22mm,
            y radius=5mm
        },
        U'/.style={scale=0.6},
        bottom arc/.style={
            start angle=-180,
            end angle=0,
            horizontal arc
        },
        upper arc/.style={
            start angle=0,
            end angle=180,
            horizontal arc
        }
    }
}

\def\drawU;{
    \foreach \yy in {-1,0,1}{
        \draw (-1, \yy) arc[bottom arc];
    }
    \draw[invisible]
        (1, 0) arc[upper arc]
        (1, -1) arc[upper arc]
        ;
    \draw (1, 1) arc[upper arc];
    \node[right] at (-1, -0.7) {$U$};

    \coordinate (A) at (-1.5, 0, 1);
    \coordinate (B) at (-1.5, 0, -1);
    \coordinate (Am1) at (A -| {-1, 0});
    \coordinate (Ap1) at (A -| {1,0});
    \coordinate (Bm1) at (B -| {-1, 0});
    \coordinate (Bp1) at (B -| {1, 0});
    \coordinate (R) at (2, 0);
    \coordinate (AR) at (A -| R);
    \coordinate (BR) at (B -| {1.5,0});
    \coordinate (O) at (1, 0);

    \draw
        (-1, 0) -- (-1, 1)
        (1, 0) -- (1, 1)
        (-1, -1) -- (Am1)
        (1, -1) -- (Ap1)
        ;
    \draw
        (AR) -- (A) -- (B) -- (Bm1)
        (Bp1) -- (BR)
        ;
    \draw[invisible]
        (Am1) -- (-1, 0)
        (Ap1) -- (1, 0)
        (Bm1) -- (Bp1)
        ;

        \node[above right=2mm of A] {$L$};
        \fill
            (0, 0) circle (2pt)
            (O) circle (2pt) node[below right] {$(1, 0)$}
            (Q) circle (2pt) node[left] {$q$};
}

\begin{document}
\maketitle
\begin{abstract}
    Consider a holomorphic foliation with singularities of a $2$-dimensional complex manifold.
    In this article we prove a new sufficient condition for this foliation to have countably many homologically independent complex limit cycles.
    In particular, if all leaves of a foliation are dense in the phase space, and it has a complex hyperbolic singular point, then it has infinitely many homologically independent complex limit cycles.
\end{abstract}

\section{Introduction}
\label{sec:intro}
The second part of Hilbert's 16th problem asks about the number and location of limit cycles of a polynomial vector field in ${\mathbb R}^2$.
In 1950s, I.~Petrovski and E.~Landis proposed \cite{PL55,PL57} to solve this problem by extending the vector field to the complex domain.
They conjectured that a complex polynomial vector field in ${\mathbb C}^2$ has finitely many complex limit cycles satisfying some additional assumptions.
This conjecture turned out to be false.
In \cite{Il69-genus:en} Yu. Ilyashenko provided an example of a complex foliation that has infinitely many complex limit cycles satisfying these assumptions.
Then he discovered \cite{Il78:en} that a generic complex polynomial vector field in ${\mathbb C}^2$ of a given degree $n\geq 2$ has infinitely many complex limit cycles.
For a~more detailed review see \cite{Il02}.
This article contains a simpler construction of complex limit cycles.

Our construction works in more general settings.
Let ${\mathcal F}$ be a holomorphic foliation with singularities of a 2-dimensional complex manifold~${\mathcal M}$.
Recall a few definitions.
\begin{definition}
    Let $L$ be a leaf of ${\mathcal F}$.
    A non-trivial free homotopy class $[\gamma ]$, $\gamma :S^1\to L$, is called a \emph{complex limit cycle} if the holonomy along its representative $\gamma $ is non-identical, and is called an \emph{identical cycle} otherwise.
\end{definition}
The holonomy maps along different representatives are conjugated, so it is not important which representative we choose.
\begin{definition}
    A~set of complex limit cycles of a~foliation ${\mathcal F}$ is called \emph{homologically independent}, if for any leaf $L$ all the cycles located on this leaf are linearly independent in $H_1(L)$.
\end{definition}
\begin{definition}
     A~singular point~$P$ of a vector field is called \emph{complex hyperbolic}, if its linearization has two non-zero eigenvalues, and their ratio $\lambda $ is not real.
 \end{definition}
Poincaré Normalization Theorem (see, e.g., \cite[Theorem 5.5, p.~62]{IYbook}) implies that a vector field is linearizable in some neighborhood of its complex hyperbolic singular point.
In particular, the corresponding foliation is given by $wdz=\lambda zdw$ in appropriate local coordinates $(z, w)$.

\begin{maintheorem}
    % Make \nameref{thm:main} work
    \phantomsection
    \expandafter\def\csname @currentlabelname\endcsname{Main Theorem}
    \label{thm:main}
    Let ${\mathcal F}$ be a holomorphic foliation with singularities of a complex 2-dimensional manifold~${\mathcal M}$.
    Suppose that it has a complex hyperbolic singularity at $P$.
    Let $U\ni P$ be an open bidisc in a linearization chart for ${\mathcal F}$ near $P$.

    Let $L$ and $L'$ be the local separatrices of~$P$ in~$U$.
    Let ${\hat L}$ be the leaf that includes~$L$.
    Suppose that ${\hat L}$ returns to $U$ not (or not only) along $L'$, i.e. ${\hat L}\cap U\nsubseteq L\cup L'$.
    Then ${\mathcal F}$ possesses a countable set of homologically independent complex limit cycles.
\end{maintheorem}
\begin{remark}
    The local leaves of all points of $U$ have $P$ in their closures, hence the assumptions of the \nameref{thm:main} do not depend on a particular choice of $U$.
\end{remark}

This theorem is mainly motivated by the study of polynomial foliations of ${\mathbb C}^2$.
We shall list corollaries of \nameref{thm:main} related to this study, and review the state of the art in this field in \autoref{sec:cor} below.
The next corollary immediately follows from the \nameref{thm:main}.
\begin{corollary}
    \label{cor:general}
    Let ${\mathcal F}$ be a holomorphic foliation with singularities of a complex $2$-dimensional manifold ${\mathcal M}$.
    Suppose that all leaves of ${\mathcal F}$ are dense in ${\mathcal M}$, and ${\mathcal F}$ possesses a complex hyperbolic singular point.
    Then ${\mathcal F}$ possesses a countable set of homologically independent complex limit cycles.
\end{corollary}
\begin{proof}
    Let $P$ be a complex hyperbolic singular point of ${\mathcal F}$.
    Let $U$, $L$, $L'$, ${\hat L}$ be as in the \nameref{thm:main}.
    Since ${\hat L}$ is dense in ${\mathcal M}$, it visits the open set $U\setminus (L\cup L')$.
    Hence, the \nameref{thm:main} implies that ${\mathcal F}$ possesses infinitely many homologically independent complex limit cycles.
\end{proof}

\section[Corollaries of the \nameref{thm:main}]{Corollaries of the \nameref{thm:main} concerning foliations of ${\mathbb C}^2$ and ${\mathbb C}P^2$}
\label{sec:cor}
\subsection{Classes \texorpdfstring{${\mathcal A}_n$}{𝒜ₙ} and \texorpdfstring{${\mathcal B}_n$}{ℬₙ}}
Let ${\mathcal A}_n$ be the class of foliations of ${\mathbb C}^2$ given by polynomial vector fields of degree at most $n$.
Let ${\mathcal B}_n$ be the class of foliations of ${\mathbb C}P^2$ given by a polynomial vector field of degree at most $n$ in each affine chart.
Geometrically, ${\mathcal B}_n$ is the class of foliations ${\mathcal F}$ of ${\mathbb C}P^2$ such that ${\mathcal F}$ has $n-1$ tangencies to a generic line $l\subset {\mathbb C}P^2$.

A foliation ${\mathcal F}\in {\mathcal A}_n$ can be extended to a holomorphic foliation of ${\mathbb C}P^2$.
Direct computation (see, e.g., \cite[Section 25A]{IYbook}) shows that the extension has degree at most $n+1$ in any other affine chart, so ${\mathcal A}_{n-1}\subset {\mathcal B}_n\subset {\mathcal A}_n$.
For a generic ${\mathcal F}\in {\mathcal A}_n$, the line $L_\infty $ at infinity is an algebraic leaf of the extended foliation, and it contains $n+1$ singular points.

Generic foliations of class ${\mathcal A}_n$ are well-studied;
the main tool here is the pseudogroup of monodromy maps along loops $\gamma \subset L_\infty $.
In particular, a generic foliation ${\mathcal F}\in {\mathcal A}_n$ possesses infinitely many complex limit cycles, and all leaves of ${\mathcal F}$ are dense in ${\mathbb C}^2$.
For more details, see \autoref{sub:An} below.

Analogous questions about a generic foliation of class ${\mathcal B}_n$ are open.

\subsection{Polynomial foliations of \texorpdfstring{${\mathbb C}P^2$}{ℂP²}}
Consider the class ${\mathcal B}_n$ of polynomial foliations of ${\mathbb C}P^2$ of projective degree $n-1$ defined above.
Though for a generic foliation ${\mathcal F}\in {\mathcal B}_n$, $n>2$, it is not known whether ${\mathcal F}$ possesses infinitely many complex limit cycles, and whether its leaves are dense in ${\mathbb C}P^2$, these properties are established for some non-empty open subsets of ${\mathcal B}_n$.

In \cite{M75:en}, B.~Mjuller studied perturbations of a foliation ${\mathcal F}\in {\mathcal B}_{2n-1}$ having a rational first integral of degree $n$.
He proved that there exists a non-empty open subset of ${\mathcal B}_{2n-1}$, $n\geq 2$, such that for each ${\mathcal F}$ from this subset, all its leaves except at most one are dense in ${\mathbb C}P^2$.

In \cite{LR03}, F.~Loray and J.~Rebelo established density of the leaves (and more properties, including ergodicity) for a foliation of ${\mathbb C}P^m$, $m>1$, close to a foliation with a specially crafted algebraic leaf.
In case $m=2$, their arguments work for an open subset ${\mathcal U}\subset {\mathcal B}_n$, $n>2$, such that the closure of ${\mathcal U}$ in ${\mathcal B}_n$ includes ${\mathcal A}_{n-1}$.

In \cite{GK-BLC}, the authors constructed an open subset ${\mathcal U}\subset {\mathcal B}_n$, $n>2$, such that ${\mathcal U}$ intersects ${\mathcal A}_{n-1}$ on a dense subset, and each foliation ${\mathcal F}\in {\mathcal U}$ possesses infinitely many homologically independent complex limit cycles.

\nameref{thm:main} allows us to construct more open subsets of ${\mathcal B}_n$ such that foliations from these subsets possess infinitely many homologically independent complex limit cycles.
In particular, \autoref{cor:general} implies that foliations from the set constructed in \cite{LR03,M75:en} satisfy assumptions of \nameref{thm:main}.

The following corollary provides us another way to construct such open set.
\begin{corollary}
    \label{cor:close-to-sep-con}
    Consider a foliation ${\mathcal F}_0\in {\mathcal B}_n$ with a separatrix connection between two complex hyperbolic singular points $P_0$ and $Q_0$. Let ${\mathcal F}\in {\mathcal B}_n$ be its perturbation such that this separatrix connection is broken: ${\mathcal F}$ has singular points $P$ and $Q$ close to  $P_0$ and $Q_0$, but no  separatrix connection close to that of ${\mathcal F}_0$.
    
    Then ${\mathcal F}$ satisfies the assumptions of the \nameref{thm:main}, hence ${\mathcal F}$ possesses infinitely many homologically independent complex limit cycles.
\end{corollary}
\begin{remark}
    We do not claim that for ${\mathcal F}_0\in {\mathcal B}_n$ with a separatrix connection, there exists ${\mathcal F}\in {\mathcal B}_n$ close to ${\mathcal F}_0$ that destroys this connection.
    We only claim that \emph{if} it exists, \emph{then} it satisfies the assumptions of \nameref{thm:main}.
\end{remark}
\begin{proof}
    Let ${\hat L}_0$ be the common separatrix of $P_0$ and $Q_0$.
    Fix a cross-section $T$ to ${\hat L}_0$, $T\cap {\hat L}_0=\set{O}$, near $P_0$, and two leafwise loops starting at $O$, $\alpha _0$ making one turn around $P$, and $\beta _0$ making one turn around $Q_0$.
    These loops define two holonomy maps $M_{\alpha _0}, M_{\beta _0}:(T, O)\to (T, O)$.
    Since $P_0$ and $Q_0$ are complex hyperbolic, we can choose the orientation of $\alpha _0$ and $\beta _0$ such that both of these holonomy maps contract to $O$.

    After a perturbation breaking the separatrix connection, singular points $P$ and $Q$ have local separatrices $L_P$ and $L_Q$, respectively.
    The holonomy maps along leafwise curves $\alpha $, $\beta $ close to $\alpha _0$ and $\beta _0$ define two holonomy maps $M_\alpha , M_\beta :T\to T$.
    These holonomy maps are defined in some neighborhoods of $O$, and contract to their unique fixed points $O_P\in L_P\cap T$ and $O_Q\in {\hat L}_Q\cap T$, where ${\hat L}_Q$ is the leaf of ${\mathcal F}$ including $L_Q$.
    Since the separatrix connection is broken, $O_P\neq O_Q$.
    Hence $M_\beta (O_P)\neq O_P$.

    Finally, $M_\beta (O_P)$ belongs to the leaf ${\hat L}_P$ that includes $L_P$, and $M_\beta (O_P)$ does not belong to the local separatrices of $P$, thus ${\mathcal F}$ satisfies the assumptions of the \nameref{thm:main}.
\end{proof}

\subsection{Polynomial foliations of \texorpdfstring{${\mathbb C}^2$}{ℂ²}}
\label{sub:An}
Consider a generic foliation ${\mathcal F}\in {\mathcal A}_n$.
Recall that ${\mathcal F}$ can be extended to a holomorphic foliation of ${\mathbb C}P^2$, and the line $L_\infty $ at infinity is an algebraic leaf of the extension.
Let $a_j$, $j=1,\dots ,n+1$ be the singular points of the extended foliation on $L_\infty $.
Denote by $\lambda _j$ the ratio of the eigenvalues of $a_j$, the one corresponding to $L_\infty $ is in the denominator.

The following theorem was initially proved by Petrovski and Landis \cite{PL55}.
Later some gaps were sealed in \cite{Pet96}, see also textbook \cite[Theorem 25.56]{IYbook}.
\begin{theorem}
    [{\cite{PL55,Pet96,IYbook}}]
    \label{thm:no-alg}
    For $n\geq 2$, a generic foliation ${\mathcal F}\in {\mathcal A}_n$ has no algebraic leaves besides the line at infinity.
    The exceptional set is included by a real algebraic subset of codimension one.
\end{theorem}

The geometric properties of a generic foliation ${\mathcal F}\in {\mathcal A}_n$ are very different from those of a generic foliation of the real plane.
We shall discuss some of the properties below.
For more details, see the survey \cite{Shch06:en}.

In 1962, M.~Khudai--Verenov \cite{KhV62:en} proved that the leaves of a generic foliation ${\mathcal F}\in {\mathcal A}_n$, $n\geq 3$, are dense in ${\mathbb C}^2$.
Namely, he proved that all leaves of~a~foliation ${\mathcal F}\in {\mathcal A}_n$ are dense, provided that ${\mathcal F}$ has no algebraic leaves, and $\lambda _j$ generate a dense subgroup in ${\mathbb C}$.
So, the exceptional set in this article has measure zero, but is dense in ${\mathcal A}_n$.

In 1978, Yu.~Ilyashenko \cite{Il78:en} proved that it is enough to require that $\lambda _j$ generate a dense subgroup in ${\mathbb C}/(2\pi i{\mathbb Z})$, so his proof works for $n=2$.
In 1984, A.~Shcherbakov \cite{Shch84:trans} proved a similar result for a thinner exceptional set, with the main assumption being the unsolvability of the monodromy group at infinity.
Also, in 1994 I.~Nakai provided another proof of the same fact.
\begin{theorem}
    [{\cite{KhV62:en,Il78:en,Shch84:trans,N94}}]
    \label{thm:density}
    For $n\geq 2$, all leaves of a generic foliation of class ${\mathcal A}_n$ are dense in ${\mathbb C}^2$.
    The exceptional set is included by a nowhere dense real analytic subset of codimension $1$.
\end{theorem}

As we mentioned \hyperref[sec:intro]{above}, the study of limit cycles of polynomial foliations ${\mathcal F}\in {\mathcal A}_n$ was motivated by Petrovskii and Landis attempt to solve the second part of Hilbert's 16th problem.
In 1978, Yu. Ilyashenko proved \cite{Il78:en} that a generic foliation ${\mathcal F}\in {\mathcal A}_n$, $n\geq 3$, has infinitely many homologically independent complex limit cycles.
As in his version of \autoref{thm:density}, the main assumption was that the subgroup of ${\mathbb C}/2\pi i{\mathbb Z}$ generated by $\lambda _j$ is dense.
This result was reinvented in 1995 \cite{G-MW95}.

In 1984, A.~Shcherbakov \cite{Shch84:trans} announced a version of Ilyashenko's theorem with the main assumption replaced by unsolvability of the monodromy group at infinity.
The proof was published in 1998 \cite{SRO98}, and works for $n\geq 3$.

In the meantime, several authors \cite{N94,BLL97,W98} provided proofs of the fact that an unsolvable subgroup of $Aut({\mathbb C}, 0)$ possesses infinitely many attracting fixed points accumulating to the origin.
This fact is closely related to the study of limit cycles for the following reason.
\begin{remark}
    \label{rem:LC-fp}
    Let $T$ be a cross-section for a foliation ${\mathcal F}$ of a two-dimensional complex manifold ${\mathcal M}$.
    Let $\gamma :[0, 1]\to {\mathcal M}$ be a leafwise path with endpoints in $T$.
    Suppose that the holonomy map $M_\gamma :(T, \gamma (0))\to (T, \gamma (1))$ extends to a holonomy map $M_\gamma :T\to T$.
    Suppose that $M_\gamma $ has an isolated fixed point $p\in T$.
    Then the curve “implementing” the equality $M_\gamma (p)=p$ is a complex limit cycle with multiplier $M_\gamma '(p)$.
\end{remark}

Recently we provided \cite{GK-BLC} another proof of the main theorem of \cite{SRO98}.
Our proof works for $n=2$, and avoids lengthy estimates of integrals.
\begin{theorem}
    [{\cite{Il78:en,SRO98,GK-BLC}}]
    \label{thm:LC}
    For $n\geq 2$, a generic foliation of class ${\mathcal A}_n$ possesses infinitely many homologically independent complex limit cycles.
    The exceptional set is included by a nowhere dense real analytic submanifold of codimension two.
\end{theorem}

\autoref{cor:general} \atpage{cor:general} shows that a slightly weaker version of \autoref{thm:LC}, with codimension one instead of two, follows from Theorems \ref{thm:no-alg} and \ref{thm:density}. 

The following corollary of \nameref{thm:main} together with \autoref{thm:no-alg} implies a version of \autoref{thm:LC} with another exceptional set (thinner than the one provided by \autoref{cor:general}) of real codimension one.
\begin{corollary}
    Let ${\mathcal F}$ be a polynomial foliation of class ${\mathcal A}_n$.
    Suppose that ${\mathcal F}$ has $n+1$ distinct complex hyperbolic points at the line at infinity, and ${\mathcal F}$ has no algebraic leaves.
    Then it satisfies the assumptions of \nameref{thm:main}.
\end{corollary}
\begin{proof}
    Let $P$ be a complex hyperbolic singular point of ${\mathcal F}$ on the line at infinity.
    Consider the separatrix of $P$ transversal to the line at infinity.
    This leaf is not algebraic, thus it accumulates to any non-singular point of the line at infinity, see \cite[Lemma 28.10]{IYbook}.
    In particular, it returns to an arbitrary neighborhood of $P$.
    Thus ${\mathcal F}$ satisfies the assumptions of \nameref{thm:main}.
\end{proof}

\subsection{Analytic foliations of \texorpdfstring{${\mathbb C}^2$}{ℂ²}}
Let ${\mathcal A}_\omega $ be the set of analytic vector fields in ${\mathbb C}^2$ with the topology of compact convergence.

The topology of the leaves of foliations given by generic vector fields $v\in {\mathcal A}_\omega $ was investigated in \cite{F06:en,K06:en} by T.~Firsova and T.~Golenishcheva--Kutuzova.
\begin{theorem}
    [\cite{F06:en}]
    \label{thm:F06}
    There exists a residual set ${\mathcal U}^F\subset {\mathcal A}_\omega $ such that for $v\in {\mathcal U}^F$, \emph{at most} countably many leaves of the foliation defined by $v$ are topological cylinders, and all other leaves are topological discs.
\end{theorem}
In \cite{K06:en}, T.~Kutuzova proved that one can drop “at most” from the statement of this theorem, i.e., a generic foliation $v\in {\mathcal A}_\omega $ possesses countably many leaves homeomorphic to cylinders.

The following corollary improves Kutuzova's theorem.
\begin{corollary}
    [cf. \cite{K06:en}]
    \label{cor:analytic}
    There exists an open dense subset ${\mathcal U}^{LC}\subset {\mathcal A}_\omega $ such that for each $v\in {\mathcal U}^{LC}$, the foliation defined by $v$ possesses infinitely many homologically independent complex limit cycles.
\end{corollary}
This corollary together with \autoref{thm:F06} implies Kutuzova's theorem.
Indeed, for $v\in {\mathcal U}^{LC}$, the corresponding foliation possesses infinitely many homologically independent complex limit cycles.
On the other hand, for $v\in {\mathcal U}^F$, all leaves are either topological discs, or topological cylinders.
Thus for $v$ from the residual set ${\mathcal U}^{LC}\cap {\mathcal U}^F$, these limit cycles are located at distinct leaves, and these leaves are topological cylinders.
In particular, the corresponding foliation has infinitely many leaves homeomorphic to cylinders.

\begin{proof}
    [\hypertarget{cor:analytic:proof}{}Proof of \autoref{cor:analytic}]
    Let ${\mathcal U}^{LC}$ be the set of vector fields $v\in {\mathcal A}_\omega $ satisfying the assumptions of the \nameref{thm:main}.
    Let us prove that it is open and dense in ${\mathcal A}_\omega $.

    \paragraph{${\mathcal U}^{LC}$ is open in ${\mathcal A}_\omega $}\hypertarget{cor:analytic:open}{}
    Consider a foliation ${\mathcal F}$ that satisfies the assumptions of the \nameref{thm:main}.
    Let $P$, $U$, $L$, $L'$, ${\hat L}$ be as in the \nameref{thm:main}.
    Let $\beta :[0, 1]\to {\hat L}$ be a path joining a point of $L$ to a point of $({\hat L}\cap U)\setminus (L\cup L')$.
    Let $K$ be the closure of an $\eps $-neighborhood of ${\overline U}\cup \beta ([0, 1])$.
    Clearly, $K$ is a compact set, and a foliation ${\mathcal F}'$ close to ${\mathcal F}$ on $K$ satisfies the assumptions of the \nameref{thm:main}.
    Thus ${\mathcal U}^{LC}$ is open in ${\mathcal A}_\omega $.

    \paragraph{${\mathcal U}^{LC}$ is dense in ${\mathcal A}_\omega $}
    Note that for all $n$, ${\mathcal U}^{LC}$ includes a dense subset of ${\mathcal A}_n$, hence $\overline{{\mathcal U}^{LC}}\supset {\mathcal A}_n$, and the union of all ${\mathcal A}_n$ is dense in ${\mathcal A}_\omega $.
    Therefore, ${\mathcal U}^{LC}$ is dense in ${\mathcal A}_\omega $ as well.
\end{proof}

\section[Independence of Limit Cycles]{A hint for the homological independence of complex limit cycles}\label{sec:indep}
Our approach includes a very simple idea which yields a great simplification in the proofs concerning independence of complex limit cycles.

Both in  \cite{SRO98,Il78:en}, the authors used the following proposition to establish the independence of complex limit cycles.

\begin{proposition}
    \label{prop:depend-01}
    Consider a tuple of cycles $c_n$, $n=1,\dots ,N$ on a surface $L$ of real dimension $2$.
    Suppose that these cycles are simple (i.e., have no self-intersections) and pairwise disjoint.
    If $[c_n]\in H_1(L)$, $n=1,\dots ,N$ are linearly dependent, then there exists a tuple $\alpha _n\in \set{-1,0,1}$, $n=1,\dots ,N$, such that $\sum_{n=1}^N \alpha _n[c_n]=0\in H_1(L)$.
\end{proposition}

\begin{corollary}
    Let ${\mathcal F}$ be an analytic foliation with singularities of the complex plane ${\mathbb C}^2$ with coordinates $(x, y)$.
    Let $(c_j)$ be a sequence of leafwise cycles of ${\mathcal F}$ such that
    \begin{enumerate}
        \item all cycles $c_j$ are simple and pairwise disjoint;
        \item the sequence
            \[
                I_j=\int_{c_j}x\,dy-y\,dx
            \]
            satisfies $|I_j|>|I_1|+\dots +|I_{j-1}|$.
    \end{enumerate}
    Then these cycles are homologically independent.
\end{corollary}
\begin{proof}
    Suppose that the cycles $[c_n]$ are homologically dependent.
    Due to \autoref{prop:depend-01}, there exists  a tuple $\alpha _n\in \set{-1,0,1}$, $n=1,\dots ,N$, such that $\sum_{n=1}^N \alpha _n[c_n]=0$.
    Then $\sum_{n=1}^N \alpha _nI_n=0$ which contradicts the inequality $|I_j|>|I_1|+\dots +|I_{j-1}|$.
\end{proof}

The proofs in  \cite{SRO98,Il78:en} thus rely on the large estimates on integrals $I_i$.
In  \cite{GK-BLC}, we suggested to use the following corollary instead:
\begin{corollary}
    [{See \cite[Lemma 9]{GK-BLC}}]
    \label{cor:cn-mu}
    Let ${\mathcal F}$ be an analytic foliation with singularities of a two-dimensional closed complex manifold ${\mathcal M}$.
    Let $c_j$ be complex limit cycles of ${\mathcal F}$ such that
    \begin{enumerate}
        \item all cycles $c_j$ are simple and pairwise disjoint;
        \item their multipliers $\mu _j=\mu (c_j)$ satisfy $|\mu _j|<1$ and $0<|\mu _j|<|\mu _1\cdots \mu _{j-1}|$.
    \end{enumerate}
     Then these cycles are homologically independent.
\end{corollary}
\begin{proof}
    Suppose that there exists  a tuple $\alpha _n\in \set{-1,0,1}$, $n=1,\dots ,N$, such that $\sum_{n=1}^N \alpha _n[c_n]=0$.
    Then $\prod (\mu _{n})^{\alpha _n}=1$ which contradicts the inequalities  $|\mu _j|<1$ and $0<|\mu _j|<|\mu _1\cdots \mu _{j-1}|$.
\end{proof}
The calculations of the multipliers of complex limit cycles are much simpler than the estimates on $I_i$, especially if the complex limit cycles are constructed as fixed points of monodromy maps, see \autoref{rem:LC-fp} above.

\section{Proof of the \nameref{thm:main}}\label{sec:proof}
In the \autoref{sub:proof}, we shall prove the \nameref{thm:main} modulo two technical lemmas, then prove these lemmas in \autoref{sub:lemmas}.
\subsection{Proof modulo technical lemmas}\label{sub:proof}
Let ${\mathcal M}$, ${\mathcal F}$, $P$, $U$, $L$, $L'$, ${\hat L}$ be as in the \nameref{thm:main}.
Recall that ${\mathcal F}$ is given by $zdw=\lambda wdz$ in $U$, and $L$, $L'$ are given by $w=0$, $z=0$, respectively.
We shall identify $U$ with the corresponding coordinate space.
After shrinking $U$ and rescaling the coordinates, we may and will assume that $U=\{\,|z|<1, |w|<1\,\}$, and ${\mathcal F}$ is given by $zdw=\lambda wdz$ in some neighborhood of ${\overline U}$.
Without loss of generality, we may and will assume that $\Im \lambda >0$, otherwise we pass to the complex conjugate coordinates.

Due to assumptions of the \nameref{thm:main}, ${\hat L}\cap {\overline U}\nsubseteq L\cup L'$.
Let $q$ be a point of $({\hat L}\cap {\overline U})\setminus (L\cup L')$.
Let $O$ be its projection to $L$.
Note that $O, q\in {\hat L}$.
Join $O$ to $q$ by a leafwise path $\beta :[0, 1]\to {\hat L}$.

We shall postpone the proof of the following lemma to \autoref{sub:lemmas}, see \autoref{fig:lem-beta}.
\begin{lemma}
    \label{lem:choice-beta}
    We can shrink $U$, rescale $z$, $w$, and choose a smooth path $\beta :[0, 1]\to {\hat L}$ without self-intersections so that
    \begin{itemize}
        \item $q=\beta (1)=(1, w_0)$, $|w_0|<1$;
        \item $\beta (t)\notin {\overline U}$ for $0<t<1$.
    \end{itemize}
\end{lemma}

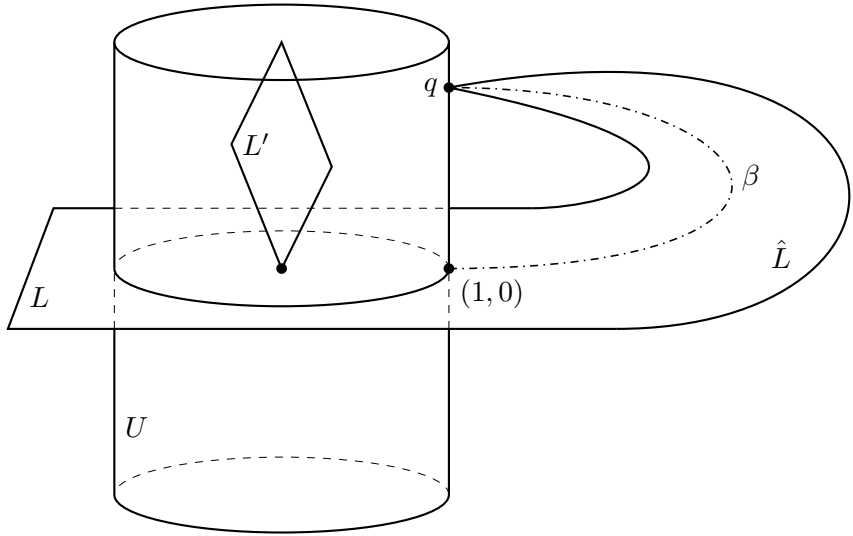
\begin{figure}[h]
    \centering
    \begin{tikzpicture}
        \coordinate (Q) at (1, 0.8);
        \drawU;
        \draw (0, 0) -- (0.3, 0.45) -- (0, 1) -- (-0.3, 0.55) node[right] {$L'$} -- cycle;
        \draw
            (AR) .. controls (AR -| {4,0}) and (4, 1.2) .. (Q) node[above left,near start] {${\hat L}$}
            (BR) .. controls (BR -| {2,0}) and (3, 0.5) .. (Q);
        \draw[beta]
            (1, 0) .. controls (3.5, 0) and (Q -| {3,0}) .. (Q) node[midway,right] {$\beta $}
            ;
    \end{tikzpicture}
    \caption{Domain $U$ and path $\beta $ provided by \autoref{lem:choice-beta}.}
    \label{fig:lem-beta}
\end{figure}

From now on, let $\beta $ be a leafwise path satisfying assertions of \autoref{lem:choice-beta}.
Consider the cross-section $T=\set{(z, w)|z=1}\subset U$.
Let $\alpha :[0, 1]\to L$ be the unit circle making one turn around $P$ in $L\setminus \set{P}$, $\alpha (t)=\left(e^{2\pi it}, 0\right)$.
Then $M_\alpha (w)=\nu w$, where $\nu =e^{2\pi i\lambda }$.
Recall that $\Im \lambda >0$, hence $|\nu |<1$.

Let $D$ be a small closed disc $O\in D\subset T$ such that the holonomy map $M_\beta :D\to T$ along $\beta $ is well-defined.
All the holonomy maps $M_n=M_\alpha ^n\circ M_\beta =\nu ^nM_\beta $ are defined in $D$.
For $n$ large enough, $M_n$ contracts in $D$, and $M_n(D)\subset D$.
Let $p_n$ be the unique fixed point of $M_n$ in $D$, $M_n(p_n)=p_n$.

Let $[c_n]$ be the complex limit cycle corresponding to $p_n$, see \autoref{rem:LC-fp}.
Let $\mu _n$ be the multiplier of $c_n$.
Since $|\mu _n|=|\nu |^n\times |M_\beta '(p_n)|\leq |\nu |^n\times \max_{w\in D}|M_\beta '(w)|\to 0$, we can choose a sequence $(n_k)$ such that $|\mu _{n_0}|<1$, and $|\mu _{n_{k+1}}|<|\mu _{n_0}\times \dots \times \mu _{n_k}|$.
Finally, if the cycles $[c_n]$ have simple and pairwise disjoint representatives $c_n$, then \autoref{cor:cn-mu} implies that the cycles $[c_{n_k}]$ are homologically independent.
Thus the following lemma completes the proof of the \nameref{thm:main}.

\begin{lemma}
    \label{lem:cn-simple-disjoint}
    In the settings introduced above, one can choose $D$ and representatives $c_n$ so that the cycles $c_n$ with sufficiently large indices are simple and pairwise disjoint.
\end{lemma}

\subsection{Proofs of technical lemmas}\label{sub:lemmas}
\begin{proof}
    [Proof of \autoref{lem:choice-beta}]
    \phantomsection\label{proof:lem:choice-beta}
    Take a leafwise path $\tilde\beta :[0, 1]\to {\hat L}$ joining $(1, 0)\in L$ to a point of $({\hat L}\cap \partial U)\setminus (L\cup L')$.
    Without loss of generality, we may and will assume that $\tilde\beta (0)=(1, 0)$, and $\tilde\beta (t)\notin {\overline U}$ for $0<t<1$.
    Indeed, otherwise we cut $\tilde\beta $ to its first return to ${\overline U}\setminus (L\cup L')$, and deform it to avoid intersections with $L\cup L'\subset U$.

    Put $\tilde\beta (1)=({\tilde z}, {\tilde w})$, then $|{\tilde z}|=1$ (“vertical” component of $\partial U$) or $|{\tilde w}|=1$ (“horizontal” component of $\partial U$).

    \begin{tcases}
        \item[$|{\tilde z}|=1$, $|{\tilde w}|<1$]
            \label{case:wlt1}
            After rotation of $z$-coordinate, we may and will assume that ${\tilde z}=1$.
            We modify an initial segment of $\tilde\beta $ so that it starts at $(1, 0)$, and immediately leaves ${\overline U}$.
            This yields $\beta $ with desired properties.
        \item[$|{\tilde w}|=1$]
            The leaf ${\hat L}$ near $\tilde\beta (1)=({\tilde z}, {\tilde w})$ is given by $z={\tilde z}e^t$, $w={\tilde w}e^{\lambda t}$, $t\in ({\mathbb C}, 0)$.
            Let us append to $\tilde\beta $ a curve of the form $\beta _\kappa =({\tilde z}e^{\tau \kappa }, {\tilde w}e^{\lambda \tau \kappa })$, $\tau \in [0, 1]$, see \autoref{fig:case-II-I}.

            Since $\lambda \notin {\mathbb R}$, we can choose $\kappa $ so that $\Re \kappa <0$ and $\Re (\lambda \kappa )<0$, hence both $|{\tilde z}e^{\tau \kappa }|$ and $|{\tilde w}e^{\lambda \tau \kappa }|$ decrease as $\tau $ increases.
            Let us shrink $U$ to $U'=\{\,|z|\leq |e^\kappa {\tilde z}|, |w|\leq 1\,\}$, and scale the coordinates appropriately.
            Then $[(|e^\kappa {\tilde z}|, 0), (1, 0)]\cup \tilde\beta \cup \beta _\kappa $ and $U'$ satisfy assumptions of \autoref{case:wlt1}.
            This case is thus reduced to \autoref{case:wlt1} above.
    \end{tcases}
\end{proof}
\begin{figure}
    \centering
    \begin{tikzpicture}
        \coordinate (Q) at (0.8, 1);
        \drawU;

        \begin{scope}
            [thin]
            \path[U']
                (1, 0) coordinate (U'r)
                (-1, 0) coordinate (U'l)
                ;
            \draw
                (0, 1) circle[U',horizontal arc]
                (U'r |- {0,-1}) -- (U'r |- A)
                (U'r) -- (U'r |- {0,1})
                (U'l |- {0,-1}) -- (U'l |- A)
                (U'l) -- (U'l |- {0,1})
                ;
            \draw[invisible]
                (U'r |- A) -- (U'r)
                (U'l |- A) -- (U'l)
                ;
            \draw[shift={(0,-1)},U'] (-1, 0) arc[bottom arc];
            \draw[invisible,shift={(0,-1)},U'] (1, 0) arc[upper arc];
            \node[right] at (U'l |- {0,-0.7}) {$U'$};
        \end{scope}

        \draw (0, 0) -- (0.3, 0.45) -- (0, 1) -- (-0.3, 0.55) node[right] {$L'$} -- cycle;
        \draw
            (AR)
                .. controls (AR -| {4,0})  and (4, 0.7) .. (3, 1.2) node[above left,near start] {${\hat L}$}
                .. controls (2, 1.7) and (1, 1.5) .. (Q)
            (BR)
                .. controls (BR -| {2,0}) and (2.5, 0.7) .. (2, 1)
                .. controls (1.5, 1.3) and (1, 1.2) .. (Q);
        \draw[beta]
            (U'r) -- (1.75, 0)
                .. controls (3, 0) and (3, 0.7) .. (2.5, 1) node[midway,below right] {$\tilde\beta $}
                .. controls (2, 1.3) and (1, 1.4) .. (Q)
                .. controls (Q |- {0,0.8}) .. (0.5, 0.7) node[midway,below right] {$\beta _\kappa $}
                coordinate (BKE);
        \fill
            (BKE) circle (2pt)
            (U'r) circle (2pt)
            ;
    \end{tikzpicture}
    \caption{Reduction of Case II to Case I in the \hyperref[proof:lem:choice-beta]{proof} of \autoref{lem:choice-beta}.}
    \label{fig:case-II-I}.
\end{figure}

\begin{proof}
    [Proof of \autoref{lem:cn-simple-disjoint}]
    Recall that $\beta $ satisfies assertions of \autoref{lem:choice-beta}.
    Let us choose lifts $\gamma _\beta [w]$, $w\in (T, O)$ of $\beta $ to nearby leaves of ${\mathcal F}$ such that
    \begin{itemize}
        \item $\gamma _\beta [w]$ is a simple path joining $(1, w)$ to $(1, M_\beta (w))$;
        \item the paths $\gamma _\beta $ are pairwise disjoint;
        \item $\gamma _\beta [w](t)\notin {\overline U}$ for $0<t<1$.
    \end{itemize}
    The lifts will satisfy the first two assumptions almost automatically, and we can choose them to satisfy the last assumption, because $\beta (t)\notin {\overline U}$ for $0<t<1$. 

    Next, we choose $D$, $O\in D\subset T$ such that 
    \begin{itemize}
        \item for $w\in D$, $\gamma _\beta [w]$ is defined and satisfies the assumptions listed above;
        \item $M_\beta $ is univalent in $D$;
        \item for $w, w'\in D$ we have
            \begin{equation}
                \label{eqn:D-small}
                |\nu |<\left|\frac{M_\beta (w)}{M_\beta (w')}\right|<|\nu |^{-1}.
            \end{equation}
    \end{itemize}

    Let $c_n$ be the union of the curve $\gamma _\beta [p_n]$ and the spiral $\gamma _{\alpha ^n}[M_\beta (p_n)]$ joining $M_\beta (p_n)$ to $\nu ^nM_\beta (p_n)=p_n$, where $\gamma _{\alpha ^n}[w]$ is given by
    \begin{align}
        \label{eqn:gamma-alpha-n}
        \gamma _{\alpha ^n}[w](t)&=\left(e^{2\pi it}, we^{2\pi i\lambda t}\right), &t&\in [0, n].
    \end{align}
    Clearly, $c_n$ is a representative of $[c_n]$.
    Let us prove that these representatives are simple and pairwise disjoint.

    For $m\neq n$ we have $\frac{M_\beta (p_n)}{p_n}=\nu ^n\neq \nu ^m=\frac{M_\beta (p_m)}{p_m}$, hence $p_n\neq p_m$.
    Therefore, the paths $\gamma _\beta [p_n]$, $n>N$, are simple and pairwise disjoint.
    The bidisc ${\overline U}$ includes all the spirals \eqref{eqn:gamma-alpha-n}, so the spirals $\gamma _{\alpha ^n}[M_\beta (p_n)]$ cannot intersect the paths $\left.\gamma _\beta [w]\right|_{(0, 1)}$.
    Thus it is enough to prove that the spirals $\gamma _{\alpha ^n}[M_\beta (p_n)]$ are simple pairwise disjoint paths.

    Consider two pairs $(n, t)$, $t\in [0, n]$, and $(m, s)$, $s\in [0, m]$, such that the corresponding points of the spirals $\gamma _{\alpha ^n}[M_\beta (p_n)]$, $\gamma _{\alpha ^m}[M_\beta (p_m)]$ coincide,
    \[
        \left( e^{2\pi it}, e^{2\pi i\lambda t}M_\beta (p_n) \right) = \left( e^{2\pi is}, e^{2\pi i\lambda s}M_\beta (p_m) \right).
    \]
    Since $e^{2\pi it}=e^{2\pi is}$, the difference $k=t-s$ is integer.
    Then $M_\beta (p_m)=e^{2\pi i\lambda (t-s)}M_\beta (p_n)=\nu ^kM_\beta (p_n)$, thus $\frac{M_\beta (p_m)}{M_\beta (p_n)}=\nu ^k$.
    Due to \eqref{eqn:D-small}, this is possible only for $k=0$, hence $t=s$, and $M_\beta (p_m)=M_\beta (p_n)$.
    Recall that $M_\beta $ is univalent in $D$, therefore $p_m=p_n$.
    Since $p_m\neq p_n$ for $m\neq n$, we have $m=n$.
    Finally, $(n, t)=(m, s)$, so the spirals $\gamma _{\alpha ^n}[M_\beta (p_n)]$ are simple and pairwise disjoint.
\end{proof}
\section*{Acknowledgements}
We started investigating complex limit cycles during our visit to Mexico City and Cuernavaca in 2014.
We are very grateful to UNAM (Mexico) and HSE (Moscow) for supporting this visit.
We are grateful to Laura Ortiz\footnote{\MRauthor{356784}} and Alberto Verjovsky\footnote{\orcid{0000-0001-6631-9637}, \MRauthor{177975}} for the invitation and fruitful discussions during this visit.

We are also grateful to our science advisor Yulij Ilyashenko\footnote{\orcid{0000-0003-1087-5903}, \MRauthor{190226}} for permanent encouragement, and to Victor Kleptsyn\footnote{\MRauthor{751650}} for useful remarks and suggestions.
\printbibliography
\end{document}